\documentclass[a4paper,12pt]{article}
\usepackage[utf8]{inputenc}

\usepackage{mathrsfs}
\usepackage{graphicx}
\usepackage{enumerate}
\usepackage{multicol}
\usepackage{color}
\usepackage{amsmath,amssymb,amscd}
\usepackage{amsthm}
\usepackage{mathabx}

\usepackage{pdfpages}
\usepackage{hyperref}

\usepackage{times}
\usepackage[varg]{txfonts}

\usepackage[top=1in, bottom=1in, left=0.75in, right=0.75in]{geometry}

\theoremstyle{plain}
\newtheorem{thm}{Theorem}
\newtheorem{lem}{Lemma}
\newtheorem{cor}{Corollary}

\theoremstyle{definition}
\newtheorem{dfn}{Definition}

\theoremstyle{remark}
\newtheorem{rem}{Remark}

\newtheorem*{ackn}{Acknowledgment}

\title{A functional equation for monomial functions}
\author{Eszter Gselmann and Mehak Iqbal}

\begin{document}
\maketitle

\begin{abstract}
Let $\mathbb{F}\subset \mathbb{K}$ be fields with characteristic zero, $n$ be a positive integer and $\kappa\in \mathbb{K}$. In this paper, we determine those monomials $f\colon \mathbb{F}\to \mathbb{K}$ of degree $n$ for which 
\[
f(x^{2})= \kappa\cdot  x^{n}f(x)
\]
holds for all $x\in \mathbb{F}$. We show that similar to the classical results, where additive functions were considered, the monomial functions in the equation can be represented with the aid of homomorphisms and higher-order derivations. 

\end{abstract}

\section{Introduction}

In this paper, $\mathbb{F}\subset \mathbb{K}$ will be fields with characteristic zero, $n$ will be a positive integer and $\kappa$ will belong to the field $\mathbb{K}$. 

The study of polynomial equations, including those involving non-additive functions, is a complex and challenging topic in both algebra and the theory of functional equations. In particular, when considering solutions in terms of generalized monomials, such as quadratic functions, the question of whether specific forms can describe these functions arises.
In this paper, we focus on addressing this problem by examining the equation 
\[
f(x^{2})= \kappa \cdot  x^{n}f(x)
\qquad 
\left(x\in \mathbb{F}\right)
\]
for the unknown monomial function $f\colon \mathbb{F}\to \mathbb{K}$. 
Our goal is to present a method for determining the solutions.

In the context of algebra and functional equations, the description of additive functions fulfilling certain polynomial equations has attracted significant interest. According to some classical results, for the additive function $a$, and specific polynomials $P$ and $Q$, the solution to the functional equation
\[a(P(x)) = Q(x, a(x))\]
necessarily takes the form of a homomorphism, derivation, or some combination of these forms. 
These questions can be generalized in several ways: instead of one, we can consider several additive functions, and instead of additive functions, we can also consider generalized polynomials. 
Assume that $n$ and $k$ are positive integers, $P_{i, j}\in \mathbb{F}[x]$ and $P\in \mathbb{K}[z, x_{1, 1}, \ldots, x_{n, k}]$ are given polynomials for $i=1, \ldots, n; j=1, \ldots, k$. Suppose further that $f_{1}, \ldots, f_{n}\colon \mathbb{F}\to \mathbb{K}$ are generalized monomials (of possibly different degree) such that 
\[
\tag{$\ast$}\label{ast}
P(x, f_{1}(P_{1, 1}(x)), \ldots, f_{1}(P_{1, k}(x)), \ldots, f_{n}(P_{n, 1}(x)), \ldots, f_{n}(P_{n, k}(x)))=0
\]
holds for all $x\in \mathbb{F}$. The question is the same, namely what we can say about the unknown functions in the equation. Is it true, for example, that they can be represented using homomorphisms and (higher-order) derivatives?

In the case where the unknown functions in the equation are additive, many results are known, e.g 
\cite{Eba15, Eba17, Eba18, EbaRieSah, GseKis24a, GseKis24b, GseKisVin18, GseKisVin19}. 
Moreover, if the unknown functions in the equation are assumed to be higher-order (typically second-order) monomials, we mention the papers \cite{Amo20, BorGar22, BorGar18, BorGar23AM, BorMen23AM, BorMen23AMS}. 

Our study, inspired by \cite{BorGar18}, seeks to build upon the work of the authors who investigated real quadratic functions $f$ that satisfy the equation
\[
f(x^2)=K\cdot  x^{2} f(x) 
\qquad 
\left(x\in \mathbb{R}\right)
\]
Their research laid the foundation for our further exploration of functional equations involving monomial functions. Assuming that the functions involved are monomials (possibly of different degrees), we started the systematic examination of polynomial equation in the papers  \cite{GseIqb23, GseIqb24}.  Our paper delves into a natural extension of the aforementioned equation, as we set out to uncover the properties of monomial functions $f\colon \mathbb{F}\to \mathbb{K}$ of degree $n$ that satisfy the equation
\[
f(x^{2})= \kappa \cdot  x^{n} f(x)
\qquad 
\left(x\in \mathbb{F}\right). 
\]
As we will see, similar to the additive and quadratic functions, the monomial functions in the equation in some cases can be represented with the aid of homomorphisms and higher-order derivations.

\subsection*{Preliminaries}

The most important concepts and statements related to generalized polynomial functions are described below. The monograph \cite{Sze91} serves as an indispensable resource for this, providing a comprehensive examination of the key notions and statements that will be necessary in the second section.

\begin{dfn}
 Let $G, S$ be commutative semigroups, $n\in \mathbb{N}$ and let $A\colon G^{n}\to S$ be a function.
 We say that $A$ is \emph{$n$-additive} if it is a homomorphism of $G$ into $S$ in each variable.
 If $n=1$ or $n=2$ then the function $A$ is simply termed to be \emph{additive}
 or \emph{bi-additive}, respectively.
\end{dfn}

The \emph{diagonalization} or \emph{trace} of an $n$-additive
function $A\colon G^{n}\to S$ is defined as
 \[
  A^{\ast}(x)=A\left(x, \ldots, x\right) = A([x]_{n})
  \qquad
  \left(x\in G\right).
 \]
As a direct consequence of the definition each $n$-additive function
$A\colon G^{n}\to S$ satisfies
\[
 A(x_{1}, \ldots, x_{i-1}, kx_{i}, x_{i+1}, \ldots, x_n)
 =
 kA(x_{1}, \ldots, x_{i-1}, x_{i}, x_{i+1}, \ldots, x_{n})
 \qquad 
 \left(x_{1}, \ldots, x_{n}\in G\right)
\]
for all $i=1, \ldots, n$, where $k\in \mathbb{N}$ is arbitrary. The
same identity holds for any $k\in \mathbb{Z}$ provided that $G$ and
$S$ are groups, and for $k\in \mathbb{Q}$, provided that $G$ and $S$
are linear spaces over the rationals. For the diagonalization of $A$
we have
\[
 A^{\ast}(kx)=k^{n}A^{\ast}(x)
 \qquad
 \left(x\in G\right).
\]

The above notion can also be extended for the case $n=0$ by letting 
$G^{0}=G$ and by calling $0$-additive any constant function from $G$ to $S$. 

One of the most important theoretical results concerning
multiadditive functions is the so-called \emph{Polarization
formula}, that briefly expresses that every $n$-additive symmetric
function is \emph{uniquely} determined by its diagonalization under
some conditions on the domain as well as on the range. Suppose that
$G$ is a commutative semigroup and $S$ is a commutative group. The
action of the {\emph{difference operator}} $\Delta$ on a function
$f\colon G\to S$ is defined by the formula
\[\Delta_y f(x)=f(x+y)-f(x)
\qquad
\left(x, y\in G\right). \]
Note that the addition in the argument of the function is the
operation of the semigroup $G$ and the subtraction means the inverse
of the operation of the group $S$.

\begin{thm}[Polarization formula]\label{Thm_polarization}
 Suppose that $G$ is a commutative semigroup, $S$ is a commutative group, $n\in \mathbb{N}$.
 If $A\colon G^{n}\to S$ is a symmetric, $n$-additive function, then for all
 $x, y_{1}, \ldots, y_{m}\in G$ we have
 \[
  \Delta_{y_{1}, \ldots, y_{m}}A^{\ast}(x)=
  \left\{
  \begin{array}{rcl}
   0 & \text{ if} & m>n \\
   n!A(y_{1}, \ldots, y_{m}) & \text{ if}& m=n.
  \end{array}
  \right.
 \]

\end{thm}

\begin{cor}
 Suppose that $G$ is a commutative semigroup, $S$ is a commutative group, $n\in \mathbb{N}$.
 If $A\colon G^{n}\to S$ is a symmetric, $n$-additive function, then for all $x, y\in G$
 \[
  \Delta^{n}_{y}A^{\ast}(x)=n!A^{\ast}(y).
\]
\end{cor}

\begin{lem}
\label{mainfact}
  Let $n\in \mathbb{N}$ and suppose that the multiplication by $n!$ is surjective in the commutative semigroup $G$ or injective in the commutative group $S$. Then for any symmetric, $n$-additive function $A\colon G^{n}\to S$, $A^{\ast}\equiv 0$ implies that
  $A$ is identically zero, as well.
\end{lem}

\begin{dfn}
 Let $G$ and $S$ be commutative semigroups, a function $p\colon G\to S$ is called a \emph{generalized polynomial} from $G$ to $S$ if it has a representation as the sum of diagonalizations of symmetric multi-additive functions from $G$ to $S$. In other words, a function $p\colon G\to S$ is a 
 generalized polynomial if and only if, it has a representation 
 \[
  p= \sum_{k=0}^{n}A^{\ast}_{k}, 
 \]
where $n$ is a nonnegative integer and $A_{k}\colon G^{k}\to S$ is a symmetric, $k$-additive function for each 
$k=0, 1, \ldots, n$. In this case we also say that $p$ is a generalized polynomial \emph{of degree at most $n$}. 

Let $n$ be a nonnegative integer, functions $p_{n}\colon G\to S$ of the form 
\[
 p_{n}= A_{n}^{\ast}, 
\]
where $A_{n}\colon G^{n}\to S$ is symmetric, $n$-additive function are the so-called \emph{generalized monomials of degree $n$}. 
\end{dfn}

\begin{rem}
 Generalized monomials of degree $0$ are constant functions and generalized monomials of degree $1$ are additive functions. 
 Furthermore, generalized monomials of degree $2$ are termed as \emph{quadratic functions}. 
\end{rem}

In what follows we will work on fields (with characteristic zero). Accordingly, let $\mathbb{F}, \mathbb{K}$ be fields with characteristic zero such that $\mathbb{F}\subset \mathbb{K}$, and $n$ be a positive integer. Let further 
\[
\mathscr{M}_{n}(\mathbb{F}, \mathbb{K})= 
\left\{ f\colon \mathbb{F}\to \mathbb{K}\, \vert \, \text{$f$ is a monomial of degree $n$} \right\}. 
\]

As we will see in the next section, in many cases, the solutions of the investigated functional equation can be obtained using higher-order derivatives. Thus, we also recall this concept, the interested reader will find further results in \cite{Rei98, UngRei98}. 

\begin{dfn}
 Let $\mathbb{F}\subset \mathbb{\mathbb{K}}$ be fields with characteristic zero. The identically zero map is the only \emph{derivation of order zero}. For each $n\in \mathbb{N}$, an additive mapping 
 $f\colon \mathbb{F}\to \mathbb{K}$ is termed to be a \emph{derivation of order $n$}, if there exists $B\colon \mathbb{F}\times \mathbb{F}\to \mathbb{K}$ such that 
 $B$ is a bi-derivation of order $n-1$ (that is, $B$ is a derivation of order $n-1$ in each variable) and 
 \[
  f(xy)-xf(y)-f(x)y=B(x, y) 
  \qquad 
  \left(x, y\in \mathbb{F}\right). 
 \]
 The set of derivations of degree $n$ of the field $\mathbb{F}$ will be denoted by $\mathscr{D}_{n}(\mathbb{F}, \mathbb{K})$. 
\end{dfn}

\section{Results}

\begin{thm}\label{thm_n}
Let $n$ be a given positive integer, $\kappa\in \mathbb{K}$ and $f\in \mathscr{M}_{n}(\mathbb{F}, \mathbb{K})$ such that 
\begin{equation}\label{Eq_monn}
f(x^{2})= \kappa \cdot x^{n} f(x)
\end{equation}
holds for all $x\in \mathbb{F}$. 
Then 
\begin{enumerate}[(i)]
    \item if $\kappa \notin \left\{ 2^{k}\, \vert \, k=0, 1, \ldots, n\right\}$, then $f$ is identically zero, 
    \item if $\kappa=1$, then 
    \[
    f(x)= f(1)\cdot x^{n} 
    \qquad 
    \left(x\in \mathbb{F}\right), 
    \]
    \item if $\kappa=2$, then there exists an $a\in \mathscr{D}_{2n-1}(\mathbb{F}, \mathbb{K})$ such that 
    \[
f(x)= \sum_{j=1}^{n}\lambda_{n, j}x^{n-j}a(x^{j})
\qquad 
\left(x\in \mathbb{F}\right)
\]
holds with appropriate constants $\lambda_{1}, \ldots, \lambda_{n}\in \mathbb{K}$, 
\item if $\kappa=2^{n}$, then there exists a symmetric and $n$-additive mapping $A_{n}\colon \mathbb{F}^{n}\to \mathbb{K}$ such that 
\begin{multline*}
\sum_{\sigma \in \mathscr{S}_{n+1}}
\left\{
A_{n}(x_{\sigma(1)}\cdot x_{\sigma(2)}, x_{\sigma(3)},..., x_{\sigma(n+1)})\right. 
\\
\left. - x_{\sigma(1)} \cdot A_{n}(x_{\sigma(2)}, \ldots, x_{\sigma(n+1)})-x_{\sigma(2)}\cdot A_{n}(x_{\sigma(1)}, \ldots, x_{\sigma(n+1)})\right\}=0
\\
\left(x_{1}, \ldots, x_{n+1}\in \mathbb{F}\right)
\end{multline*}
and 
\[
f(x)= A_{n}(x, \ldots, x) \qquad 
\left(x\in \mathbb{F}\right). 
\]
holds. 
\end{enumerate}
    
\end{thm}

\begin{proof}
Since $f\in \mathscr{M}_{n}(\mathbb{F}, \mathbb{K})$, there exists a uniquely determined symmetric $n$-additive function $A_{n}\colon \mathbb{F}^{n}\to \mathbb{K}$ such that 
\[
f(x)= A_{n}(x, \ldots, x)= A_{n}([x]_{n}) 
\qquad 
\left(x\in \mathbb{F}\right). 
\]
In terms of the mapping $A_{n}$, equation \eqref{Eq_monn} reads as 
\[
A_{n}(x^{2}, \ldots, x^{2})-\kappa x^{n} A_{n}(x, \ldots, x)=0 
\qquad 
\left(x\in \mathbb{F}\right). 
\]
If we write $x+1$ in place of $x$ in the above identity and expand the terms, 
\begin{multline}\label{Eq_nadd}
\sum_{\substack{\alpha_{1}+\alpha_{2}+\alpha_{3}=n\\ \alpha_{1}, \alpha_{2}, \alpha_{3} \geq 0}} 
\binom{n}{\alpha_{1}, \alpha_{2}, \alpha_{3}}A_{n}\left([x^{2}]_{\alpha_{1}}, [2x]_{\alpha_{2}}, [1]_{\alpha_{3}}\right)
\\
-\kappa \left(\sum_{\substack{\beta_{1}+\beta_{2}=n\\ \beta_{1}, \beta_{2}\geq 0}} \binom{n}{\beta_{1}, \beta_{2}}x^{\beta_{1}}\right)\cdot \left(\sum_{\substack{\gamma_{1}+\gamma_{2}=n \\ \gamma_{1}, \gamma_{2}\geq 0}} \binom{n}{\gamma_{1}, \gamma_{2}}A_{n}([x]_{\gamma_{1}}, [1]_{\gamma_{2}})\right)=0 
\end{multline}
follows for all $x\in \mathbb{F}$. 
Observe that the left-hand side of this equation is a generalized polynomial of degree $2n$, which is identically zero. Thus all of its monomial terms should vanish. 
If $\alpha_{1}, \alpha_{2}, \alpha_{3}$ are nonnegative and $\alpha_{1}+\alpha_{2}+ \alpha_{3} =n$, then the degree of the mapping 
\[
x\longmapsto \binom{n}{\alpha_{1}, \alpha_{2}, \alpha_{3}}A_{n}\left([x^{2}]_{\alpha_{1}}, [2x]_{\alpha_{2}}, [1]_{\alpha_{3}}\right) 
\]
is $2\alpha_{1}+\alpha_{2}$. 
Similarly, if $\beta_{1}, \beta_{2}, \gamma_{1}, \gamma_{2}$ are nonnegative integers with $\beta_{1}+\beta_{2}=n$ and $\gamma_{1}+\gamma_{2}=n$, then the degree of the generalized monomial 
\[
x\longmapsto 
\kappa \binom{n}{\beta_{1}, \beta_{2}}x^{\beta_{1}}\cdot \binom{n}{\gamma_{1}, \gamma_{2}}A_{n}([x]_{\gamma_{1}}, [1]_{\gamma_{2}})
\]
is $\beta_{1}+\gamma_{1}$. 

Computing the zero-degree terms in \eqref{Eq_nadd}, 
\[
(1-\kappa)A_{n}([1]_{n})=0
\]
follows. So $\kappa=1$ or $A_{n}([1]_{n})= f(1)=0$. 
The first-degree terms of the generalized polynomial in equation \eqref{Eq_nadd} must vanish, thus we have 
\begin{multline*}
\binom{n}{ 0, 1, n-1}A_{n}([2x_{1}], [1]_{n-1})
\\
-\kappa \left\{
\binom{n}{0, n}\binom{n}{ 1, n-1}A_{n}([x]_{1}, [1]_{n-1})+\binom{n}{1, n-1}\binom{n}{0, n}xA_{n}([1]_{n})
\right\}
=0, 
\end{multline*}
that is, 
\[
(2-\kappa)A_{n}([x]_{1}, [1]_{n-1})-\kappa x A_{n}([1]_{n})=0
\]
for all $x\in \mathbb{F}$. 
If $\kappa=1$, then this leads to 
\[
A_{n}([x]_{1}, [1]_{n-1})= A_{n}([1]_{n})\cdot x 
\qquad 
\left(x\in \mathbb{F}\right). 
\]
If $\kappa\neq 1$, then due to the previous step $A_{n}([1]_{n})=0$,  the above identity reduces to 
\[
(2-\kappa)A_{n}([x]_{1}, [1]_{n-1}) =0
\qquad 
\left(x\in \mathbb{F}\right). 
\]
So $\kappa=2$ (and in this case,  we do not have any information for the values $A_{n}([x]_{1}, [1]_{n-1})$), or $\kappa \neq 1, 2$ and then $A_{n}([x]_{1}, [1]_{n-1}) =0$ for all $x\in \mathbb{F}$. 
In general, if $k=0, 1, \ldots, 2n$, we have to distinguish two cases, depending on whether $k$ is even or odd.

If \emph{$k$ is even}, then the $k$\textsuperscript{th}-degree term in \eqref{Eq_nadd} is 
\begin{multline*}
\sum_{l=0}^{\frac{k}{2}}\binom{n}{l, k-2l, n-k}A_{n}\left([x^{2}]_{l}, [2x]_{k-2l}, [1]_{n-k}\right)
\\
-\kappa \sum_{m=0}^{k}\binom{n}{m, n-m} \binom{n}{k-m, n-k+m}x^{m}A_{n}([x]_{k-m}, [1]_{n-k-m})
\end{multline*}
If \emph{$k$ is odd}, then the only difference is that in the first expression, we must not sum up to $\frac{k}{2}$, but to $\frac{k-1}{2}$, i.e.,
\begin{multline*}
\sum_{l=0}^{\frac{k-1}{2}}\binom{n}{l, k-2l, n-k}A_{n}\left([x^{2}]_{l}, [2x]_{k-2l}, [1]_{n-k}\right)
\\
-\kappa \sum_{m=0}^{k}\binom{n}{m, n-m} \binom{n}{k-m, n-k+m}x^{m}A_{n}([x]_{k-m}, [1]_{n-k-m})
\end{multline*}

This means that for all $k=0, 1, \ldots, 2n$, we have 
\begin{multline*}
    \sum_{l=0}^{\lfloor\frac{k}{2}\rfloor}\binom{n}{l, k-2l, n-k}A_{n}\left([x^{2}]_{l}, [2x]_{k-2l}, [1]_{n-k}\right)
\\
-\kappa \sum_{m=0}^{k}\binom{n}{m, n-m} \binom{n}{k-m, n-k+m}x^{m}A_{n}([x]_{k-m}, [1]_{n-k-m})
\end{multline*}

By selecting the members with indices $l=0$ and $m=0$ from the sums, we get that 
\begin{multline}\label{Eq_nadd_k}
\binom{n}{k, n-k}(2^{k}-\kappa)A_{n}([x]_{k}, [1]_{n-k})
\\
=
- \sum_{l=1}^{\lfloor\frac{k}{2}\rfloor}\binom{n}{l, k-2l, n-k+l}A_{n}\left([x^{2}]_{l}, [2x]_{k-2l}, [1]_{n-k+l}\right)
\\
+\kappa \sum_{m=1}^{k}\binom{n}{m, n-m} \binom{n}{k-m, n-k+m}x^{m}A_{n}([x]_{k-m}, [1]_{n-k-m})
\end{multline}
for all $x\in \mathbb{F}$. This identity indicates that the values of the left-hand $k$-variable function are determined by the right-hand side, with the aid of the at most $(k-1)$-variable functions, for all $k=0, 1, \ldots, n$. 

First assume that for all $k=0, 1, \ldots, n$, we have $\kappa \neq 2^{k}$. Then for all $k=0, 1, \ldots, n$, 
\[
A_{n}([x]_{k}, [1]_{n-k})=0 
\qquad 
\left(x\in \mathbb{F}\right), 
\]
thus especially, 
\[
f(x)= A_{n}([x]_{n})=0 
\qquad 
\left(x\in \mathbb{F}\right). 
\]
We will show this by induction on $k$.
Above, in the comparison of zero-degree terms, we have already seen that 
\[
(1-\kappa)A_{n}([1]_{n})=0. 
\]
Since $\kappa \neq 1$ in this case, $A_{n}([1]_{n})=0$.  Therefore the statement holds for $k=0$. Assume now that there exists a positive integer, such that the statement holds for all indices less than $k$. In other words, suppose that for all $l=0, 1, \ldots, k-1$
\[
A_{n}([x]_{l}, [1]_{n-l})=0 
\qquad 
\left(x\in \mathbb{F}\right). 
\]
Since the left-hand side is a generalized monomial of degree $l$, which is identically zero, the symmetric, $l$-additive mapping,  defined uniquely by it, must also be identically zero. This means however that 
\[
A_{n}(x_{1}, \ldots, x_{l}, [1]_{n-l})=0 
\]
holds for all $x_{1}, \ldots, x_{l}\in \mathbb{F}$. 
This means that the right-hand side of equation \eqref{Eq_nadd_k} is identically zero, from which 
\[
\binom{n}{k, n-k}(2^{k}-\kappa)A_{n}([x]_{k}, [1]_{n-k})=0 
\qquad 
\left(x\in \mathbb{F}\right)
\]
can be deduced. So the statement holds also for $k$. Summing up, if for all $k=0, 1, \ldots, n$, we have $\kappa \neq 2^{k}$,  then 
\[
A_{n}([x]_{k}, [1]_{n-k})=0 
\qquad 
\left(x\in \mathbb{F}\right), 
\]
thus especially, 
\[
f(x)= A_{n}([x]_{n})=0 
\qquad 
\left(x\in \mathbb{F}\right). 
\]
We now turn to discussing the case $\kappa =1$. In this case, we will also use equation \eqref{Eq_nadd_k} by choosing different indices $k$. We show that for all $k=0, 1, \ldots, n$, we have 
\[
A_{n}([x]_{k}, [1]_{n-k})= f(1) \cdot x^{k} 
\qquad 
\left(x\in \mathbb{F}\right). 
\]
For $k=0$, this holds true trivially since $A_{n}([1]_{n})= f(1)$. 
Assume now that there exists a positive $k$ less or equal to $n$ such that the statement holds for $l=0, \ldots, k-1$, i.e., 
\[
A_{n}([x]_{l}, [1]_{n-l})= f(1) \cdot x^{l} 
\qquad 
\left(x\in \mathbb{F}\right). 
\]
Since $A_{n}$ is a symmetric and $n$-additive mapping, for all $l=0, \ldots, n$, the mappings 
\[
\mathbb{F}\ni x\longmapsto A_{n}([x]_{l}, [1]_{n-l})
\]
is symmetric and $l$-additive. Therefore, the stronger statement
\[
A_{n}(x_{1}, \ldots, x_{l}, [1]_{n-l})= f(1) \cdot x_{1}\cdots x_{l}
\qquad 
\left(x_{1}, \ldots, x_{l}\in \mathbb{F}\right). 
\]
is also true for all $l=0, \ldots, k-1$. 
In this case, however, equation \eqref{Eq_nadd_k} takes the form 
\begin{multline*}
\binom{n}{k, n-k}(2^{k}-1)A_{n}([x]_{k}, [1]_{n-k})
\\
=
- \sum_{l=1}^{\lfloor\frac{k}{2}\rfloor}\binom{n}{l, k-2l, n-k+l}A_{n}\left([x^{2}]_{l}, [2x]_{k-2l}, [1]_{n-k}\right)
\\
+\sum_{m=1}^{k}\binom{n}{m, n-m} \binom{n}{k-m, n-k+m}x^{m}A_{n}([x]_{k-m}, [1]_{n-k-m})
\\
=
- \sum_{l=1}^{\lfloor\frac{k}{2}\rfloor}\binom{n}{l, k-2l, n-k+l}f(1) \cdot x^{k}
\\
+\sum_{m=1}^{k}\binom{n}{m, n-m} \binom{n}{k-m, n-k+m}x^{m}\cdot f(1)\cdot  x^{k-m}
\\
= \binom{n}{k, n-k}(2^{k}-1) \cdot f(1) x^{k} \qquad 
\left(x\in \mathbb{F}\right), 
\end{multline*}
showing that the statement also holds for $k$. 
Note, however, that this means (with $k=n$) that
\[
A_{n}([x]_{n})= f(1)\cdot x^{n} 
\qquad 
\left(x\in \mathbb{F}\right), 
\]
that is, 
\[
f(x)= f(1)\cdot x^{n} 
\qquad 
\left(x\in \mathbb{F}\right). 
\]
We continue with the case $\kappa=2$. Recall that in this case, we have already seen that $A_{n}([1]_{n})=0$ and equation \eqref{Eq_nadd_k} for $k=1$ does not contain any information for the values of the mapping $x\longmapsto A_{n}([x]_{1}, [1]_{n-1})$. Consider the additive function $a$ defined on $\mathbb{F}$ by 
\[
a(x)= A_{n}([x]_{1}, [1]_{n-1}) 
\qquad 
\left(x\in \mathbb{F}\right). 
\]
We show that for all $k=1, \ldots, n$, 
\[
A_{n}([x]_{k}, [1]_{n-k})= \sum_{l=1}^{k}\lambda_{k, l}x^{k-l} a(x^{l}) 
\]
holds with some appropriate constants $\lambda_{k, 1}, \ldots, \lambda_{k, l}\in \mathbb{K}$. 
Due to the definition of the function $a$, this statement is true for $k=1$. Assume an index $k\geq 2$ exists now, such that the above statement holds for all index $l=1, \ldots, k-1$. 
Note from the induction hypothesis it follows that we have 
\[
A_{n}(x_{1}, \ldots, x_{l}, [1]_{n-l})= \frac{1}{l!}\sum_{\sigma\in \mathscr{S}_{l}}\sum_{m=1}^{l}
\lambda_{l, m}x_{\sigma(m+1)}\cdots x_{\sigma(l)}\cdot a(x_{\sigma(1)}\cdots x_{\sigma(m)})
= \frac{1}{l!}\sum_{m=1}^{l}\Delta_{x_{1}, \ldots, x_{l}}x^{l-m}a(x^{m})
\]
for all $x_{1}, \ldots, x_{l}$ and for all $l=1, \ldots, k-1$. 
Therefore, we have 
\[
A_{n}\left([x^{2}]_{l}, [2x]_{k-2l}, [1]_{n-k+l}\right)
=
\frac{1}{(k-l)!}\sum_{m=1}^{l}\Delta_{[x^{2}]_{l}, [2x]_{k-2l}}y^{l-m}a(y^{m})
= \sum_{j=1}^{k}\widetilde{\lambda_{k, j}} x^{k-j}a(x^{j})
\]
for all $x_{1}, \ldots, x_{l}$ and for all $l=1, \ldots, k-1$.
Due to identity \eqref{Eq_nadd_k} and the induction hypothesis, we have 
\begin{multline*}
\binom{n}{k, n-k}(2^{k}-\kappa)A_{n}([x]_{k}, [1]_{n-k})
\\
=
- \sum_{l=1}^{\lfloor\frac{k}{2}\rfloor}\binom{n}{l, k-2l, n-k+l}A_{n}\left([x^{2}]_{l}, [2x]_{k-2l}, [1]_{n-k+l}\right)
\\
+\kappa \sum_{m=1}^{k}\binom{n}{m, n-m} \binom{n}{k-m, n-k+m}x^{m}A_{n}([x]_{k-m}, [1]_{n-k-m})
\\
=
\sum_{j=1}^{k}\widetilde{\lambda_{k, j}}x^{k-j}a(x^{j})
+\kappa \sum_{m=1}^{k}\binom{n}{m, n-m} \binom{n}{k-m, n-k+m}x^{m}\sum_{j=1}^{k-m}\lambda_{k-m, j}x^{k-m-j}a(x^{j})
\\
=
\sum_{j=1}^{k}\widetilde{\lambda_{k, j}}x^{k-j}a(x^{j})
+\kappa \sum_{m=1}^{k}\sum_{j=1}^{k-m}\binom{n}{m, n-m} \binom{n}{k-m, n-k+m}\lambda_{k-m, j}x^{k-j}a(x^{j})
\\
= \sum_{j=1}^{k}\lambda^{\ast}_{k, j}x^{k-j}a(x^{j})
\end{multline*}
for all $x\in \mathbb{F}$. From this, we obtain that the statement also holds for $k$. 
Observe that this means that 
\[
f(x)=A_{n}([x_{n}])= \sum_{j=1}^{n}\lambda_{n, j}x^{n-j}a(x^{j})
\qquad 
\left(x\in \mathbb{F}\right)
\]
holds with appropriate constants $\lambda_{1}, \ldots, \lambda_{n}$. 
Using this representation and equation \eqref{Eq_monn}, 
identity 
\[
 \sum_{j=1}^{n}\lambda_{n, j}x^{2n-2j}a(x^{2j})-\sum_{j=1}^{n}2\lambda_{n, j}x^{2n-j}a(x^{j})
 =0 
 \qquad 
 \left(x\in \mathbb{F}\right)
\]
follows for the additive function $a$. In view of Corollary 4 of \cite{GseKisVin18}, this means that $a\in \mathscr{D}_{2n-1}(\mathbb{F}, \mathbb{K})$. 

Let us consider the case $\kappa= 2^{n}$. We show that in this case 
\[
A_{n}([x]_{k}, [1]_{n-k})=0 
\qquad 
\left(x\in \mathbb{F}\right)
\]
holds for all $k=0, \ldots, n-1$. Since in case $\kappa=2^{n}$, we have $A_{n}([1]_{n})=0$, the statement holds for $k=0$. Assume now that there exists a $k\in \left\{1, \ldots, n \right\}$ such that the statement holds for all $l= 0, \ldots, k-1$, that is, we have 
\[
A_{n}([x]_{l}, [1]_{n-l})=0 
\qquad 
\left(x\in \mathbb{F}\right)
\]
for all $l=0, \ldots, k-1$. Since the left-hand side of this equation is the trace of a symmetric and $l$-additive function, we have 
\[
A_{n}\left(x_{1}, \ldots, x_{l}, [1]_{n-l}\right)=0 
\qquad 
\left(x_{1}, \ldots, x_{l}\in \mathbb{F}\right)
\]
for all $l=0, \ldots, k-1$. 
Using this, and equation \eqref{Eq_nadd_k}, we obtain that 
\begin{multline*}
    \binom{n}{k, n-k}(2^{k}-\kappa)A_{n}([x]_{k}, [1]_{n-k})
\\
=
- \sum_{l=1}^{\lfloor\frac{k}{2}\rfloor}\binom{n}{l, k-2l, n-k+l}A_{n}\left([x^{2}]_{l}, [2x]_{k-2l}, [1]_{n-k+l}\right)
\\
+\kappa \sum_{m=1}^{k}\binom{n}{m, n-m} \binom{n}{k-m, n-k+m}x^{m}A_{n}([x]_{k-m}, [1]_{n-k-m})
=0
\end{multline*}
for all $x\in \mathbb{F}$. Indeed, since $l+(k-2l)\leq k-1$ and $k-m\leq k-1$ holds for all $l=1, \ldots, \lfloor\frac{k}{2}\rfloor$ and for all $m=1, \ldots, k$, the mappings 
\[
\mathbb{F}\ni x \longmapsto A_{n}\left([x^{2}]_{l}, [2x]_{k-2l}, [1]_{n-k+l}\right) 
\quad 
\text{and}
\quad 
\mathbb{F}\ni x \longmapsto A_{n}([x]_{k-m}, [1]_{n-k-m})
\]
are identically zero. 

Since the left-hand side of equation \eqref{Eq_nadd_k} vanishes for $\kappa=2^{n}$, we have no information for the mapping $x\longmapsto A_{n}([x]_{n})$. 

Using this, and computing the $(n+1)$\textsuperscript{st}-degree terms in \eqref{Eq_nadd}, we deduce that 
\[
n\cdot A_{n}(x^2, [2x]_{n-1})-n\cdot x \cdot 2^{n} \cdot A_{n}([x]_{n})=0 
\qquad 
\left(x\in \mathbb{F}\right), 
\]
that is, 
\[
A_{n}([x^{2}]_{1}, [x]_{n-1})-2x\cdot A_{n}([x]_{n})=0
\qquad 
\left(x\in \mathbb{F}\right). 
\]
The left-hand side of this equation (as a mapping of the variable $x$) is the trace of a symmetric and $(n+1)$-additive function. However, the $(n+1)$-additive mapping is necessarily identically zero if the trace vanishes. Thus we have 
\[
\sum_{\sigma \in \mathscr{S}_{n+1}}
\left\{
A_{n}(x_{\sigma(1)}\cdot x_{\sigma(2)}, x_{\sigma(3)},... x_{\sigma(n+1)})- x_{\sigma(1)} A_{n}(x_{\sigma(2)}, \ldots, x_{\sigma(n+1)})-x_{\sigma(2)}A_{n}(x_{\sigma(1)}, \ldots, x_{\sigma(n+1)})\right\}=0
\]
for all $x_{1}, \ldots, x_{n+1}\in \mathbb{F}$. 
\end{proof}

\begin{rem}
Notice that the above theorem does not state anything about the cases when $\kappa \in \left\{2^{k}\, \vert \, k=3, 4, \right. $ $\left.  \ldots, n-1\right\}$. In addition, we make a new conjecture that in the case of $k=2$, the order of the higher-order derivative included in the production of the function $f$ can be reduced. This conjecture is supported by our results for the case $n=3$, which can be found in the next section. 
\end{rem}

\subsection*{Results on the special case $n=3$}

As a supplement to the result in the previous chapter, we will deal with the special case $n=3$ below. Compared to Theorem \ref{thm_n}, we can prove two more things. On the one hand, we show that in the case $\kappa=2$, the order of the higher-order derivation appearing in the representation of the function $f$ is at most $3$, not at most $5$. On the other hand, we show that in the case $\kappa=4$ the function $f$ is identically zero. 

First, we prove the following lemma, which will help in part (ii) of Theorem \ref{thm_3}.

\begin{lem}\label{lemma_der3}
Let $a\colon \mathbb{F}\to \mathbb{K}$ be an additive mapping such that 
\[
2\,a\left(x^6\right)-9\,x^2\,a\left(x^4\right)-4\,x^3\,a\left(x^3
 \right)+36\,x^4\,a\left(x^2\right)-36\,x^5\,a\left(x\right)=0 
 \qquad 
 \left(x\in \mathbb{F}\right). 
\]
Then $a\in \mathscr{D}_{3}(\mathbb{F}, \mathbb{K})$. 
\end{lem}

\begin{proof}
Let $a\colon \mathbb{F}\to \mathbb{K}$ be an additive function such that 
\[
2\,a\left(x^6\right)-9\,x^2\,a\left(x^4\right)-4\,x^3\,a\left(x^3
 \right)+36\,x^4\,a\left(x^2\right)-36\,x^5\,a\left(x\right)=0 
 \qquad 
 \left(x\in \mathbb{F}\right). 
\]
Applying the operator $\Delta_{1}$ to both sides of this equation, we get that 
\begin{multline*}
4\,a\left(x^5\right)-6\,x\,a\left(x^4\right)+7\,a\left(x^4\right)-
 16\,x^2\,a\left(x^3\right)-28\,x\,a\left(x^3\right)+44\,x^3\,a\left(
 x^2\right)
 \\
 +42\,x^2\,a\left(x^2\right)-36\,x^4\,a\left(x\right)-28\,x
 ^3\,a\left(x\right)=0
\end{multline*}
for all $x\in \mathbb{F}$. Observe that the left-hand side of this equation is a generalized polynomial of degree $5$, that has fifth-degree and also fourth-degree monomial terms that should vanish simultaneously. Thus collecting the fourth-degree terms, we arrive to 
\[
7\,a\left(x^4\right)-28\,x\,a\left(x^3\right)+42\,x^2\,a\left(x^2\right)-28\,x^3\,a\left(x\right)=0
\qquad 
\left(x\in \mathbb{F}\right)
\]
that is to 
\[
a\left(x^4\right)-4\,x\,a\left(x^3\right)+6\,x^2\,a\left(x^2\right)
 -4\,x^3\,a\left(x\right)=0
 \qquad 
\left(x\in \mathbb{F}\right). 
\]
In view of \cite[Corollary 2]{GseKisVin18} this means that $a\in \mathscr{D}_{3}(\mathbb{F}, \mathbb{K})$. 
\end{proof}

\begin{thm}\label{thm_3}
Let $\kappa\in \mathbb{K}$ be arbitrarily fixed and $f\in \mathscr{M}_{3}(\mathbb{F}, \mathbb{K})$ be a monomial for which 
\begin{equation}\label{eq_basic3}
f(x^{2})= \kappa \cdot  x^{3} f(x)
\end{equation}
holds for all $x\in \mathbb{F}$. 
Then the following cases are possible 
\begin{enumerate}[(i)]
    \item $\kappa=1$ and then 
    \[
    f(x)= f(1)\cdot x^{3} 
    \qquad 
    \left(x\in \mathbb{F}\right). 
    \]
\item  $\kappa=2$ and then there exists a $d\in \mathscr{D}_{3}(\mathbb{F}, \mathbb{K})$ such that 
\[
f(x)= d\left(x^3\right)-{\frac{9\,x\,d\left(x^2\right)}{2}}+9\,x^2\,d \left(x\right) 
\qquad 
\left(x\in \mathbb{F}\right). 
\]
\item $\kappa=8$  and then there exists a symmetric and $3$-additive mapping $D_{3}$ for which 
\begin{multline*}
\sum_{\sigma \in \mathscr{S}_{4}} \left[D_{3}(x_{\sigma(1)}x_{\sigma(2)}, x_{\sigma(3)}, x_{\sigma(4)})-x_{\sigma(1)}D(x_{\sigma(2)}, x_{\sigma(3)}, x_{\sigma(4)})-x_{\sigma(2)}D(x_{\sigma(1)}, x_{\sigma(3)}, x_{\sigma(4)})\right]
=0
\\
\left(x_{1}, x_{2}, x_{3}, x_{4}\in \mathbb{F}\right)
\end{multline*}
such that 
\[
f(x)= D_{3}(x, x, x) 
\qquad 
\left(x\in \mathbb{F}\right)
\]
\item $\kappa \in \mathbb{K}\setminus \left\{1, 2, 8\right\}$ and $f$ is identically zero. 
\end{enumerate}
\end{thm}

\begin{proof}
The cases $\kappa \in \mathbb{K}\setminus \left\{ 2, 4\right\}$ immediately follow from Theorem \ref{thm_3}. Therefore, it is sufficient to deal only with the cases $\kappa \in \left\{ 2, 4\right\}$. 
In case $n=3$ equation \eqref{Eq_nadd} reads as 
\begin{multline}\label{Eq_3add}
    \sum_{\alpha_{1}+\alpha_{2}+\alpha_{3}=3}\binom{3}{\alpha_{1}, \alpha_{3}, \alpha_{3}}A_{3}([x^2]_{\alpha_{1}}, [2x]_{\alpha_{2}}, [1]_{\alpha_{3}}) 
    \\
    -\kappa (x^3+3x^2+3x+1) \sum_{\alpha_{1}+\alpha_{2}=3}\binom{3}{\alpha_{1}, \alpha_{2}}A_{3}([x]_{\alpha_{1}}, [1]_{\alpha_{2}})
    =0
    \quad 
    \left(x\in \mathbb{F}\right). 
\end{multline}
At first, we consider the case $\kappa=2$. Collecting the constant terms we get that $f(1)=A_{3}(1, 1, 1)=0$. Further, computing the first-degree monomial terms, we do not get any information for the values of $A_{3}(x, 1, 1)$. 
The second-degree monomial terms must also vanish, thus
\begin{multline*}
  \binom{3}{1, 0, 2}A_{3}([x^2]_{1}, [1]_{2})+\binom{3}{0, 2, 1}A_{3}([2x]_{2}, [1]_{1})-3\kappa x^2 \binom{3}{0, 3}A_{3}([1]_{3})
  \\
  -3\kappa x \binom{3}{1, 2}A_{3}([x]_{1}, [1]_{2})-\kappa \binom{3}{2, 1}A_{3}([x]_{2}, [1]_{1})=0, 
\end{multline*}
or after some rearrangement
\[
(4-\kappa)A_{3}(x, x, 1)+A_{3}(x^2, 1, 1)-\kappa x^{2}A_{3}(1, 1, 1)-3\kappa x A_{3}(x, 1, 1)=0
\]
for all $x\in \mathbb{F}$. If $\kappa=2$, then we have 
    \[
    2A_{3}(x, x, 1)+A_{3}(x^{2}, 1, 1)-6xA_{3}(x, 1, 1)=0 
    \qquad 
    \left(x\in \mathbb{F}\right), 
    \]
    where the identity $A_{3}(1, 1, 1)=0$ was also used. Thus, 
    \[
    2A_{3}(x, x, 1)+a(x^2)-6xa(x)=0
    \]
    holds with the additive function $a\colon \mathbb{F}\to \mathbb{K}$, where 
    \[
    a(x)= A_{3}(x, 1, 1) 
    \qquad 
    \left(x\in \mathbb{F}\right). 
    \]
Finally, the fact that the third-degree terms should vanish, yields
\begin{multline*}
12A_{3}(x^2, x, 1)+8A_{3}(x, x, x)-\kappa x^{3}A_{3}(1, 1, 1)-9\kappa x^{2} A_{3}(x, 1, 1)
\\
-9\kappa x A_{3}(x, x, 1)-\kappa A_{3}(x, x, x)=0 
\qquad
\left(x\in \mathbb{F}\right), 
\end{multline*}
that is, 
\begin{multline}\label{id_third}
(8-\kappa) A_{3}(x, x, x)+12 A_{3}(x^{2}, x, 1)-9\kappa x A_{3}(x, x, 1)
\\
-9\kappa x^{2}A_{3}(x, 1, 1)-\kappa x^{3}A_{3}(1, 1, 1)=0
\qquad 
\left(x\in \mathbb{F}\right). 
\end{multline}

If $\kappa=2$,  then equation \eqref{id_third} furnishes 
\[
A_{3}(x, x, x)= a\left(x^3\right)-{{9\,x\,a\left(x^2\right)}\over{2}}+9\,x^2\,a
 \left(x\right)
\]
for all $x\in \mathbb{F}$. This means that the function $f$ can be represented as 
\[
f(x)= a\left(x^3\right)-{{9\,x\,a\left(x^2\right)}\over{2}}+9\,x^2\,a
 \left(x\right)
\]
for all $x\in \mathbb{F}$. Since $f$ satisfies equation \eqref{eq_basic3}, the additive function $a$ necessarily fulfills 
\[
2\,a\left(x^6\right)-9\,x^2\,a\left(x^4\right)-4\,x^3\,a\left(x^3
 \right)+36\,x^4\,a\left(x^2\right)-36\,x^5\,a\left(x\right)=0 
 \qquad 
 \left(x\in \mathbb{F}\right). 
\]
This however, in view of Lemma \ref{lemma_der3} implies that  $a\in \mathscr{D}_{3}(\mathbb{F})$.

Finally, in case $\kappa=4$, we show that $f$ is identically zero. Similarly, as in case $\kappa=2$, we determine the first, second, and third-degree terms in \eqref{Eq_3add}. 
Computing the first-degree terms, 
\[
A_{3}(x, 1, 1)=0
\qquad 
\left(x\in \mathbb{F}\right)
\]
follows. The second-degree terms should also vanish. Thus we have 
\[
(4-\kappa)A_{3}(x, x, 1)+A_{3}(x^2, 1, 1)-\kappa x^{2}A_{3}(1, 1, 1)-3\kappa x A_{3}(x, 1, 1)=0
\]
for all $x\in \mathbb{F}$, yielding no information for the values $A_{3}(x, x, 1)$ in the case $\kappa=4$. 
Further, if we compute the third-degree terms, we obtain that 
\[
A_{3}(x, x, x)= 9x A_{3}(x, x, 1)-3A_{3}(x^2, x, 1) = 9xB(x, x)-3B(x^2, x)
\qquad 
\left(x\in \mathbb{F}\right), 
\]
where the symmetric and bi-additive function $B\colon \mathbb{F}^{2}\to \mathbb{K}$ is defined through 
\[
B(x, y)= A_{3}(x, y, 1) 
\qquad 
\left(x, y\in \mathbb{F}\right). 
\]
This means that the values of function $B$ completely determine function $A_3$. Indeed, we have 
\[
A_{3}(x, y, z)
=
3\left[xB(y, z)+yB(x, z)+zB(x, y)\right]-\left[B(xy, z)+B(xz, y)+B(yz, x)\right]
\quad 
\left(x, y, z\in \mathbb{F}\right). 
\]
Reformulating equation \eqref{eq_basic3} with the aid of the mapping $B$, 
\[
    0= f(x^{2})-4x^{3}f(x)
    =
    -3B(x^4, x^2)+9x^2B(x^2, x^2)+12x^{3}B(x^2, x)-36x^{4}B(x, x), 
\]
that is, 
\[
    B(x^4, x^2)-3x^2B(x^2, x^2)-4x^{3}B(x^2, x)+12x^{4}B(x, x)=0
\]
follows for all $x\in \mathbb{F}$. The left side, as a function of the variable $x$, i.e., 
\[
g(x)=  B(x^4, x^2)-3x^2B(x^2, x^2)-4x^{3}B(x^2, x)+12x^{4}B(x, x)
\]
is a monomial of degree $6$, which is identically zero. Thus $\tau_{1}g$ is a generalized polynomial of degree at most $6$. Computing the fourth degree monomial terms of $\tau_{1}g$, we get that 
\begin{equation}\label{eq_kappa4_fourth}
3B(x^{2}, x^{2})-36xB(x^{2}, x)+8B(x^{3}, x)+36x^{2}B(x, x)=0 
\qquad 
\left(x\in \mathbb{F}\right). 
\end{equation}
On the other hand, if we collect the fourth-degree terms in \eqref{Eq_3add}, 
\[
A_{3}(x^2, x, x)-3x^{2}A_{3}(x, x, 1)+xA_{3}(x, x, x)=0, 
\]
expressing this, with the aid of the function $B$, 
\begin{equation}\label{eq_kappa4_fourth2}
-2B(x^3, x)-B(x^2, x^2)+9xB(x^2, x)-9x^{2}B(x, x)=0
\end{equation}
follows for all $x\in \mathbb{F}$. Combining equations \eqref{eq_kappa4_fourth} and \eqref{eq_kappa4_fourth2}, 
\[
B(x^2, x^2)=0
\]
can be concluded for all $x\in \mathbb{F}$. Since the left-hand side of this identity is the trace of the symmetric and $4$-additive mapping 
\[
B_{4}(x_1, x_2, x_3, x_4)=\frac{1}{3} \left[B(x_1x_2, x_3x_4)+B(x_1x_3, x_2x_4)+B(x_1x_4, x_2x_3) \right]
\qquad 
\left(x_1, x_2, x_3, x_4\in \mathbb{F}\right), 
\]
we have 
\[
0=B_{4}(x, y, 1, 1)= \frac{B(xy, 1)}{3}+\frac{2 B(x, y)}{3}= \frac{2}{3}B(x, y)
\]
for all $x, y\in \mathbb{F}$. So $B$ is identically zero. Therefore 
\[
f(x)=A_{3}(x, x, x)= 9xB(x, x)-3B(x^2, x)=0 
\qquad 
\left(x\in \mathbb{F}\right), 
\]
as stated. 
\end{proof}

\begin{ackn}
The research of E.~Gselmann has been supported by project no.~K134191 that has been
implemented with the support provided by the National Research, Development and Innovation Fund of Hungary, financed under the K{\_}20 funding scheme. This paper was supported by the János Bolyai Research Scholarship of the Hungarian Academy of Sciences. 
\end{ackn}

\vspace{1.5cm}

\noindent
\textbf{Eszter Gselmann} \\
Department of Analysis\\
University of Debrecen\\
P.O. Box 400\\
H-4002 Debrecen\\
Hungary\\
e-mail: \href{mailto:gselmann@science.unideb.hu}{gselmann@science.unideb.hu}\\
ORCID: \href{https://orcid.org/0000-0002-1708-2570}{0000-0002-1708-2570}
\vspace{1cm}

\noindent
\textbf{Mehak Iqbal}\\
Doctoral School of Mathematical and Computational Sciences\\
University of Debrecen\\
P.O. Box 400\\
H-4002 Debrecen\\
Hungary\\
e-mail: \href{mailto:iqbal.mehak@science.unideb.hu}{iqbal.mehak@science.unideb.hu}\\

\end{document}